\documentclass[a4paper,12pt,leqno,twoside]{article}
\usepackage[english]{babel} 
\usepackage{amsmath,amsxtra,amsthm,amssymb}
\usepackage{amsfonts}
\usepackage{latexsym}
\usepackage{xypic}
\usepackage[all]{xy}
\usepackage{stmaryrd} 
\usepackage[mathscr]{eucal}
\usepackage{eufrak}
\usepackage{epsf}
\usepackage[dvips]{graphicx}
\usepackage[active]{srcltx}

\usepackage{epsfig}

\usepackage{graphics}
\usepackage{mathrsfs} 

\usepackage{latexsym}  
\usepackage[usenames]{color}

\usepackage{fancyhdr}
\usepackage{mathrsfs} 

\usepackage{helvet}

\newcommand{\mc}[1]{\mathcal{#1}}

\renewcommand{\phi}{\varphi}
\renewcommand{\theta}{\vartheta}
\renewcommand{\rho}{\varrho}

\newcommand{\lra}{\longrightarrow}

\renewcommand{\bigr}[1]{{\big(#1\big)}}
\renewcommand{\Bigr}[1]{{\Big(#1\Big)}}

\newcommand{\ou}[3][]{\overset{{#1}}{\underset{{#2}}{{#3}}}}

\newcommand{\com}{\mathbb{C}}
\newcommand{\rea}{\mathbb{R}}

\newcommand{\integer}{\mathbb{Z}}
\newcommand{\integergz}{\mathbb{Z}_{>0}}

\newcommand{\Mod}{\mathrm{Mod}}

\newcommand{\mrm}[1]{\mathrm{#1}}
\newcommand{\mfr}[1]{\mathfrak{#1}}

\newcommand{\Op}{\mathrm{Op}}

\newcommand{\opxsac}{\mathrm{Op}^c(\xsa)}

\newcommand{\cov}[1]{\mathrm{Cov}_{sa}(#1)}

\newcommand{\xsa}{{X_{sa}}}

\newcommand{\Ho}[3][]{\mathcal{H}\mathrm{om}_{#1}(#2,#3)}
\newcommand{\mcHom}[3][]{\mathcal{H}\mathrm{om}_{#1}(#2,#3)}
\newcommand{\ho}[3][]{\mathrm{Hom}_{#1}(#2,#3)}

\newcommand{\RH}[3][]{\mathit{R}\mathcal{H}\mathit{om}_{#1}(#2,#3)}
\newcommand{\Rh}[3][]{\mathit{RHom}_{#1}(#2,#3)}
\newcommand{\M}{\mathcal{M}}
\newcommand{\ot}{\mathcal{O}^t}
\newcommand{\otxsa}{\mathcal{O}^t_\xsa}

\newcommand{\dbt}{\mathcal{D}b^t}
\newcommand{\Db}{\mathcal{D}b}

\newcommand{\dbxr}{\mathcal{D}b_{X_\rea}}

\newcommand{\dbtxsa}{\mathcal{D}b^t_{\xsa}}

\newcommand{\D}{\mathcal{D}}
\newcommand{\DX}{{\D_X}}
\newcommand{\dbdx}[1][]{D^b_{#1}(\DX)}
\newcommand{\bdc}[2][]{D^b_{#1}(#2)}
\renewcommand{\mod}{\mathrm{Mod}}
\renewcommand{\M}{\mathcal{M}}
\newcommand{\N}{\mathcal{N}}

\newcommand{\Dfinv}[1]{\mathrm D{#1}^{-1}}
\newcommand{\dfinv}[1]{{#1}_{\mathrm D}^{-1}}

\newcommand{\Dotimes}{\overset D{\otimes}}
\newcommand{\dbd}[2][]{D^b_{#1}(\D_{#2})}
\newcommand{\omt}[1]{\Omega^t_{#1}}
\newcommand{\drt}[2][]{D\!R^t_{{\D_{#1}}}{#2}}

\newcommand{\ch}{\mathrm{char}}
\newcommand{\dmod}[2][]{\mathrm{Mod}_{#1}(\mathcal D_{#2})}

\newcommand{\solt}{\mathscr{S}ol^t}

\newtheorem{thm}{Theorem}[subsection]
\newtheorem{df}[thm]{Def\mbox{}inition}

\newtheorem{prop}[thm]{Proposition}
\newtheorem{lem}[thm]{Lemma}

\newtheorem{conj}[thm]{Conjecture}

\numberwithin{equation}{section}

\newcommand{\ben}{\begin{enumerate}}
\newcommand{\een}{\end{enumerate}}

\author{Giovanni Morando}
\title{\textbf{Preconstructibility of tempered solutions of holonomic $\D$-modules}}
\date{July 2010}

\long\def\symbolfootnote[#1]#2{\begingroup%
\def\thefootnote{\fnsymbol{footnote}}\footnote[#1]{#2}\endgroup}

\newpage

\begin{document}

\maketitle

\thispagestyle{empty}

\begin{abstract}
In this paper we prove the preconstructibility of the complex of
tempered holomorphic solutions of holonomic $\D$-modules on complex
analytic manifolds. This implies the finiteness of such complex on any
relatively compact open subanalytic subset of a complex analytic
manifold. Such a result is an essential step for proving a conjecture
of M. Kashiwara and P. Schapira (\cite{ks_micro_indsheaves}) on the
constructibility of such complex.
\end{abstract}


\symbolfootnote[0]{\phantom{a}\hspace{-7mm}\textit{2010 MSC.} Primary 32C38; Secondary 32B20 32S40 14Fxx.}

\symbolfootnote[0]{\phantom{a}\hspace{-7mm}\textit{Keywords and phrases:} $\D$-modules, irregular singularities, tempered holomorphic functions, subanalytic.}

\vspace{-13mm}

\tableofcontents

\phantom{a}

\phantom{a}

\section*{Introduction}\markboth{Introduction}{Introduction}
\addcontentsline{toc}{section}{\textbf{Introduction}}
The Riemann--Hilbert correspondence is one of the more powerful
results in $\D$-module theory. It has important applications in many
fields of mathematics such as representation theory or microlocal
analysis. Such a correspondence gives an equivalence between the
category of {\emph{regular} holonomic $\D$-modules on an analytic
  manifold $X$ and the category of perverse sheaves on $X$. The former
  is of analytic nature, as it is a categorical description of certain
  linear partial differential systems. The latter is of topological
  nature as it is defined using subanalytic stratifications of $X$ and
  some combinatorial properties. Such an equivalence is realized
  through the complex of holomorphic solutions of regular holonomic
  $\D$-modules. In the proof of the Riemann--Hilbert correspondence
  given by M. Kashiwara
  (\cite{kashiwara_79,kashiwara_riemann-hilbert}) and, more
  generically, in the study of the complex of holomorphic solutions of
  a holonomic $\D$-module, the property of constructibility of such
  complex is fundamental (\cite{kashiwara_regular1}). Let us simply
  recall that a bounded complex of sheaves is constructible if its
  cohomology groups are locally constant sheaves on the strata of an
  analytic stratification and if their stalks have finite dimension.

  The study of \emph{irregular} holonomic $\D$-modules is much more
  complicated. In dimension $1$, a local irregular Riemann--Hilbert
  correspondence was proved through the works of many mathematicians
  as P.  Deligne, M. Hu\-ku\-ha\-ra, A. Levelt, B. Malgrange,
  J.-P. Ramis, Y.  Sibuya, H. Turrittin (see \cite{dmr} and
  \cite{malgrange_birkhauser}). Roughly speaking, such a
  correspondence is obtained in two steps. First one studies the
  formal structure of flat meromorphic connections (Levelt--Turrittin
  Theorem) and then one analyses the gluing of the sectorial
  asymptotic lifts of the formal structures. The higher dimensional
  case is still open. C. Sabbah conjectured the analogue of the
  Levelt--Turrittin Theorem and he proved it in some particular cases
  (\cite{sabbah_ast}). Furthermore, he deeply studied the asymptotic
  lifting property (\cite{sabbah_aif}). In \cite{mochizuki1} and
  \cite{mochizuki2}, T. Mochizuki proved Sabbah's conjecture in the
  algebraic cases. Recently (\cite{kedlaya1}, \cite{kedlaya2}),
  K. Kedlaya proved Sabbah's conjecture in the analytic case. In
  dimension $1$, the formal decomposition is given by the
  Levelt--Turrittin Theorem. In higher dimension it is not
  straightforward and one needs a finer study of the formal invariants
  of a flat meromorphic connection. This leads to introduce the notion
  of \emph{good} model which plays a central role in the cited
  works. These new results should allow a formulation of the irregular
  Riemann--Hilbert correspondence in higher dimension. Even if the
  local $1$-dimensional case is nowadays classical, it has some points
  which do not allow the passage to a global description of holonomic
  $\D$-modules or to a natural generalization to the higher
  dimensional case. In particular, the notion of the formal invariants
  of a flat meromorphic connection in dimension $1$ has a purely
  analytic nature and the higher dimensional analogues of such
  invariants need to satisfy a \emph{goodness} property.
 
  Recently (\cite{ks_micro_indsheaves}), M. Kashiwara and P. Schapira
  introduced the subanalytic site relative to $X$, denoted $\xsa$, and
  the complex of sheaves on it of tempered holomorphic functions,
  denoted $\ot_\xsa$. The study of solutions of holonomic $\D$-modules
  on complex curves with values in $\ot_\xsa$ allowed to describe
  faithfully and in a topological way the formal invariants of flat
  meromorphic connections
  (\cite{morando_tempered_solutions_formal_invariants}). Furthermore,
  Kashiwara--Schapira defined a notion of $\rea$-constructibility
  for sheaves on $\xsa$ and they conjectured the
  $\rea$-constructibility of tempered solutions of holonomic
  $\D$-modules. In \cite{morando_existence_theorem}, we proved the
  conjecture for holonomic $\D$-modules on complex curves. In the
  present article we prove the $\rea$-preconstructibility of the
  complex of tempered solutions of holonomic $\D$-modules on analytic
  manifolds. This implies the finiteness of such complex on any
  relatively compact open subanalytic subset of a complex analytic
  manifold.

  The article is organized as follows. In the first section we review
  classical results on sheaves on subanalytic sites, $\D$-modules,
  tempered solutions and elementary asymptotic decompositions of flat
  meromorphic connections. In this paper, the notion of good model and
  good decomposition is not needed. Hence, we will simply recall the
  elementary asymptotic decomposition of flat meromorphic connections
  without any goodness property. In the second section we state and
  prove our main result of $\rea$-preconstructibility of the tempered
  De Rham complex of holonomic $\D$-modules. The proof is divided in
  two parts. In the first part we prove our result for modules
  supported on a closed analytic subset $Z$ of $X$ by using induction
  on the dimension of $Z$ and reducing to the $1$-dimensional case. In
  the second part, we consider flat meromorphic connections. We start
  by proving the $\rea$-preconstructibility of the tempered De Rham
  complex of an elementary model using the results in dimension $1$
  and some properties on the pull-back of the tempered De Rham
  complex. Then, we prove the case of a generic flat meromorphic
  connections by using the results of Kedlaya, Mochizuki and Sabbah.

  Let us conclude the introduction recalling that many spaces of
  functions with growth conditions have been used in the study of
  irregular $\D$-modules. For example, in \cite{malgrange_birkhauser},
  the sheaves of holomorphic functions with asymptotic expansion or
  moderate and Gevrey growth at the origin have been studied. In
  higher dimension, sheaves of functions with moderate growth or with
  asymptotic expansion on a divisor $Z$ of $X$ have been used in
  \cite{sabbah_aif}, \cite{sabbah_ast}, \cite{hien_periods} and
  \cite{hien_periods_manifolds}. Such sheaves are defined on the real
  blow-up of $X$ along $Z$ and they do not allow to treat globally
  holonomic $\D_X$-modules. The $\rea$-preconstructibility of the
  tempered De Rham complex of an elementary model, obtained in the
  present article (see Lemma \ref{lem:finiteness_good_models}), is
  related to the results, obtained by M. Hien in \cite{hien_periods}
  and \cite{hien_periods_manifolds}, on the vanishing of the moderate
  De Rham complex of a good model. Our approach is quite different to
  Hien's one. Indeed, Hien proved his results of finiteness and
  vanishing of the moderate De Rham complex of good models by
  estimating the growth of some integrals appearing in the solutions
  on multisectors (see \cite{sabbah_aif} for the asymptotic expansion
  case).  In our case the situation is more complicated. Indeed, with
  Hien's approach we should have proved growth estimates of solutions
  on arbitrary subanalytic open sets. We used this method in dimension
  $1$ (\cite{morando_existence_theorem}) but in higher dimension the
  geometry of a subanalytic set can be much more complicated. Hence we
  adopted a different approach to prove the
  $\rea$-preconstructibility of the tempered De Rham complex of an
  elementary model. We used our results in dimension $1$, combining it
  with some formulas on the pull-back of the tempered De Rham complex
  proved in \cite{ks_indsheaves}.
  Remark that, adapting the techniques of
  \cite{honda_prelli_multi-specialization}, one can obtain holomorphic
  functions with moderate growth from tempered holomorphic functions
  in a functorial way.


  {\em Acknowledgments}: I wish to express my gratitude to P. Schapira
  for his constant encouragement in the preparation of this paper. I
  am deeply indebted with C. Sabbah for many essential discussions, I
  wish to warmly thank him here.

\section{Notations and review}  \label{recall}
\subsection{Subanalytic sites}

Let $X$ be a real analytic manifold countable at inf\mbox{}inity.

\begin{df}{\label{def:semi-sub-analytic}}
\begin{enumerate}
\item A set $Z\subset X$ is said \emph{semi-analytic at} $x\in X$ if
  the following condition is satisf\mbox{}ied. There exists an open
  neighborhood $W$ of $x$ such that $Z\cap W=\cup_{i\in I}\cap_{j\in
    J}Z_{ij}$ where $I$ and $J$ are f\mbox{}inite sets and either
  $Z_{ij}=\{y\in X;\ f_{ij}(y)>0\}$ or $Z_{ij}=\{y\in X;\
  f_{ij}(y)=0\}$ for some real-valued real analytic functions $f_{ij}$
  on $W$. Furthermore, $Z$ is said \emph{semi-analytic} if $Z$ is
  semi-analytic at any $x\in X$.
\item A set $Z\subset X$ is said \emph{subanalytic} if the following
  condition is satisf\mbox{}ied. For any $x\in X$, there exist an open
  neighborhood $W$ of $x$, a real analytic manifold $Y$ and a
  relatively compact semi-analytic set $A\subset X\times Y$ such that
  $\pi(A)=Z\cap W$, where $\pi: X\times Y\to X$ is the projection.
\end{enumerate}
\end{df}

Given $Z\subset X$, denote by $\mathring Z$ (resp. $\overline Z$,
$\partial Z$) the interior (resp. the closure, the boundary) of $Z$.

\begin{prop}[See \cite{bm_ihes}]
  Let $Z$ and $V$ be subanalytic subset of $X$. Then $Z\cup V$, $Z\cap
  V$, $\overline Z$, $\mathring Z$ and $Z\setminus V$ are
  subanalytic. Moreover the connected components of $Z$ are
  subanalytic, the family of connected components of $Z$ is locally
  f\mbox{}inite and $Z$ is locally connected at any point in $Z$.
\end{prop}

\begin{df}
  A family $\{A_\alpha\}_{\alpha\in\Lambda}$ of subanalytic subsets of
  $X$ is said a \emph{stratif\mbox{}ication of $X$} if
  $\{A_\alpha\}_{\alpha\in\Lambda}$ is locally f\mbox{}inite,
  $X=\ou{\alpha\in\Lambda}{\bigsqcup}A_\alpha$ and each $A_\alpha$ is
  a locally closed subanalytic manifold.
\end{df}

For the theory of sheaves on topological spaces we refer to
\cite{ks_som}. For the theory of sheaves on the subanalytic site that
we are going to recall now, we refer to \cite{ks_indsheaves}, see also
\cite{prelli_subanalytic_sheaves}.

We denote by $\Op(X)$ the family of open subsets of $X$. For $k$ a
commutative ring we denote by $k_X$ the constant sheaf. For a sheaf in
rings $\mc R_X$, we denote by $\Mod(\mc R_X)$ the category of sheaves
of $\mc R_X$-modules on $X$ and by $\bdc{\mc R_X}$ the bounded derived
category of $\Mod(\mc R_X)$.

Let us recall the def\mbox{}inition of the subanalytic site $\xsa$
associated to $X$. An element $U\in\Op(X)$ is an open set for $\xsa$
if it is open, relatively compact and subanalytic. The family of open
sets of $\xsa$ is denoted $\opxsac$. For $U\in\opxsac$, a subset $S$
of the family of open subsets of $U$ is said an open covering of $U$
in $\xsa$ if $S\subset\opxsac$ and, for any compact $K$ of $X$, there
exists a f\mbox{}inite subset $S_0\subset S$ such that
$K\cap(\cup_{V\in S_0}V)=K\cap U$. The set of coverings of $U$ in
$\xsa$ is denoted by $\cov U$.

We denote by $\Mod(k_\xsa)$ the category of sheaves of $k$-modules on
the subanalytic site associated to $X$. With the aim of def\mbox{}ining the category
$\Mod(k_\xsa)$, the adjective ``relatively compact'' can be omitted in
the def\mbox{}inition above. Indeed, in \cite[Remark 6.3.6]{ks_indsheaves},
it is proved that $\Mod(k_\xsa)$ is equivalent to the category of
sheaves on the site whose open sets are the open subanalytic subsets
of $X$ and whose coverings are the same as $\xsa$.

Let $\mrm{PSh}(k_\xsa)$ be the category of presheaves of
$k$-modules on $\xsa$.
Denote by $for:\Mod(k_\xsa)\to \mrm{PSh}(k_\xsa)$ the forgetful
functor which associates to a sheaf $F$ on $\xsa$ its underlying
presheaf. It is well known that $for$ admits a left adjoint
$\cdot^a:\mrm{PSh}(k_\xsa)\to \Mod(k_\xsa)$.

We denote by
$$ \rho:X\lra\xsa \ ,$$ 
the natural morphism of sites associated to
$\opxsac\lra\Op(X)$. We refer to \cite{ks_indsheaves} for the
def\mbox{}initions of the functors $\rho_*:\Mod(k_X)\lra\Mod(k_\xsa)$
and $\rho^{-1}:\Mod(k_\xsa)\lra\Mod(k_X)$ and for Proposition
\ref{prop_functors} below.
\begin{prop}\label{prop_functors}
\begin{enumerate}
\item The functor $\rho^{-1}$ is left adjoint to $\rho_*$.
\item The functor $\rho^{-1}$ has a left adjoint denoted by
  $\rho_!:\Mod(k_X)\to\Mod(k_\xsa)$.
\item The functors $\rho^{-1}$ and $\rho_!$ are exact, $\rho_*$ is
  exact on constructible sheaves.
\item The functors $\rho_*$ and $\rho_!$ are fully faithful.
\end{enumerate}
\end{prop}

Through $\rho_*$, we will consider $\Mod(k_X)$ as a subcategory of $\Mod(k_\xsa)$.

The functor $\rho_!$ is described as follows. If
$U\in\opxsac$ and $F\in\Mod(k_X)$, then $\rho_!(F)$ is the sheaf on $\xsa$
associated to the presheaf $U\mapsto F\bigr{\overline U}$. 


Now, we are going to recall the definition of
$\rea$-preconstructibility and $\rea$-constructibility for sheaves on
$\xsa$.

Denote by $\bdc[\rea-c]{\com_X}$ the full triangulated subcategory of
the bounded derived category of $\mod(\com_X)$ consisting of complexes
whose cohomology groups are $\rea$-constructible sheaves. In what
follows, for $F\in\bdc{\com_\xsa}$ and $G\in\bdc[\rea-c]{\com_X}$, we
set for short
$$ \RH[\com_X]GF:=\rho^{-1}\RH[\com_\xsa]GF\in\bdc{\com_X} \
 $$
and
$$ \Rh[\com_X]GF:=R\Gamma(X,\RH[\com_X]GF) \ . $$

\begin{df} Let $F\in\bdc{\com_\xsa}$.
  \begin{enumerate}
  \item We say that $F$ is \emph{$\rea$-preconstructible} if for any
    $G\in\bdc[\rea-c]{\com_X}$ with compact support and any
    $j\in\integer$,
$$ \dim_\com\mrm R^jHom_{\com_X}(G,F))<+\infty\ . $$
\item   We say that $F$ is
  \emph{$\rea$-constructible} if for any $G\in\bdc[\rea-c]{\com_X}$,
$$ \RH[\com_X]GF\in\bdc[\rea-c]{\com_X}\ . $$
\end{enumerate}
\end{df}

\subsection{$\D$-modules and tempered holomorphic functions}

For the general theory of $\D$-modules, we refer to
\cite{kashiwara_dmod} and \cite{bjork}. For the basic results on the
tempered De Rham functor we refer to \cite{ks_indsheaves}. For an
introduction to derived categories, we refer to \cite{ks_som}.

Let $X$ be a complex analytic manifold. We denote by $\mathcal{O}_X$ the sheaf
of rings of holomorphic functions on $X$ and by $\D_X$ the sheaf of
rings of linear partial differential operators with coefficients in
$\mathcal{O}_X$.

Given two left $\D_X$-modules $\mc M_1,\mc M_2$, we denote by $\mc
M_1\Dotimes\mc M_2$ the internal tensor product. Let us start by
recalling the following

\begin{prop}[\cite{kashiwara_dmod} Proposition 3.5]\label{prop:tensor_commutation}
Let $\mc N$ be a right $\D_X$-module, $\mc M_1, \mc M_2$ left $\D_X$-modules. Then
$$  \mc N\ou{\D_X}\otimes(\mc M_1\Dotimes\mc M_2)\simeq (\mc N\ou{\mathcal{O}_X}\otimes \mc M_1)\ou{\D_X}\otimes \mc M_2  \ .$$  
\end{prop}

Now, let $T^*X$ denote the cotangent bundle on $X$. We denote by
$\dmod[c] X$ the full subcategory of $\dmod X$ whose objects are
coherent over $\D_X$. For $\M\in\dmod[c] X$ we denote by $\ch\M$ the
characteristic variety of $\M$. Recall that $\ch\M\subset T^*X$ and
that $\mc M$ is said \emph{holonomic} if $\ch\M$ is Lagrangian. We
denote by $\dmod[h] X$ the full subcategory of $\dmod X$ consisting
of holonomic modules.

\sloppy  
We denote by $\bdc[coh]{\D_X}$ (resp. $\bdc[h]{\D_X}$) the full subcategory of
$\dbdx$ consisting of bounded complexes whose cohomology groups are
coherent (resp. holonomic) $\D_X$-modules. 
For
$\M\in\dbdx[coh]$, set $ \ch\M:= \cup_{j\in\integer}\ch H^j(\M)$. 

Let $\pi_X:T^*X\to X$ be the canonical projection, $T^*_XX$ the zero
section of $T^*X$ and $\dot T^*X:=T^*X\setminus T^*_XX$.

For $\M\in\bdc[coh]{\D_X}$, set 
$$ S(\M):=\pi_X\Bigr{\ch\M\cap\dot T^*X} \ . $$

It is well known that, if $\M\in\dbdx[h]$, then $S(\M)\neq X$ is a closed
\emph{analytic} subset of $X$. That is to say, for any $z\in S(\M)$ there exist a
neighbourhood $W$ of $z$ and finitely many functions $f_1,\ldots,
f_k\in\mathcal{O}_X(W)$ such that $S(\M)\cap W=\{x\in W;f_1(x)=\ldots=f_k(x)=0\}$.

\begin{df}
An object $\M\in\dbdx[h]$ is said \emph{regular holonomic} if, for any $x\in X$,
$$ \mathrm{RHom}_{\D_X}(\M,\mathcal{O}_{X,x})\overset{\sim}{\lra} \mathrm{RHom}_{\D_X}(\M,\widehat{\mathcal{O}}_{X,x})  \
, $$ where $\widehat{\mathcal{O}}_{X,x}$ is the $\D_{X,x}$-module of formal power
series at $x$. We denote by $\dbdx[rh]$ the full subcategory of
$\dbdx[h]$ of regular holonomic
$\D_X$-modules. 
\end{df}

Now, let $Z$ be a closed analytic subset of $X$. Let $\mc I_Z$ be the
coherent ideal consisting of the holomorphic functions vanishing on
$Z$, we set
\begin{eqnarray*}
  \Gamma_{[Z]}\mc M&:=&\ou k\varinjlim\ \mcHom[\mathcal{O}_X]{\mathcal{O}_X/\mc I^k_Z}{\mc M} \ ,\\
  \Gamma_{[X\setminus Z]}\mc M&:=&\ou k\varinjlim\ \mcHom[\mathcal{O}_X]{\mc I^k_Z}{\mc M} \ .
\end{eqnarray*}

If $S\subset X$ can be written as $S=Z_1\setminus Z_2$, for $Z_1$ and
$Z_2$ closed analytic sets, then it can be proved that the following
object is well defined

$$ \Gamma_{[S]}\mc M:=\Gamma_{[Z_1]}\Gamma_{[X\setminus Z_2]}\mc M  $$
and that $\Gamma_{[S]}$ is a left exact functor. We denote by
$R\Gamma_{[S]}$ the left derived functor of $\Gamma_{[S]}$.

\begin{thm}[\cite{kashiwara_dmod}, Theorem 3.29]
  \begin{enumerate}
  \item Let $S_1,S_2\subset X$ be difference of closed analytic subsets
    of $X$. Then $$  R\Gamma_{[S_1]}R\Gamma_{[S_2]}\mc M\simeq
    R\Gamma_{[S_1\cap S_2]} \mc M \ .$$
  \item Let $Z$ be a closed analytic subset of $X$. For any $\mc M\in\dbdx$
there exists a distinguished triangle
$$ R\Gamma_{[Z]}\mc M\lra\mc M\lra R\Gamma_{[X\setminus Z]}\mc
M\overset{+1}\lra \ . $$
  \end{enumerate}
\end{thm}

Let us also recall the following fundamental 

\begin{thm}[\cite{kashiwara_dmod} Theorem 4.30]\label{thm:kashiwara_lemma}
  Let $Z$ be a closed submanifold of $X$. Then the category of
  coherent $\D_X$-modules supported by $Z$ is equivalent to the
  category of coherent $\D_Z$-modules.
\end{thm}

Given two complex analytic manifolds $X$ and $Y$ of dimension,
respectively, $d_X$ and $d_Y$ and a holomorphic morphism $f:X\to Y$,
we denote by $\dfinv f:\dmod Y\to\dmod X$ the inverse image functor
and by $\Dfinv f:\dbd Y\to \dbd X$ its derived functor. Recall that
$f$ is said \emph{smooth} if the corresponding maps of tangent spaces
$T_xX\to T_{f(x)}Y$ are surjective for any $x\in X$. If $f$ is a
smooth map, then $\dfinv f$ is an exact functor. For Proposition
\ref{prop:comm_supp_inv_im} below, we refer to Proposition 2.5.27 of
\cite{bjork} and to Proposition 3.35 of \cite{kashiwara_dmod}.

\begin{prop}\label{prop:comm_supp_inv_im}
Let $Z\subset Y$ be an analytic set, $f:X\to Y$ a holomorphic
  map, $\mc M\in\bdc [h]{\D_Y}$.  Then
 $$  R\Gamma_{[f^{-1}(Z)]}(\mathcal{O}_X)\simeq \Dfinv f(R\Gamma_{[Z]}\mc O_Y)  \
 \textrm{ and } \  R\Gamma_{[Z]}\mc M\simeq R\Gamma_{[Z]}\mc O_Y\Dotimes\mc M
    \ .$$
 In particular, 
$$ R\Gamma_{[f^{-1}(Z)]}(\Dfinv f\mc M)\simeq \Dfinv f(R\Gamma_{[Z]}\mc M)   \ . $$
\end{prop}

Given $f\in\mathcal{O}_X$, let $Z:=f^{-1}(0)$. We denote by $\mathcal{O}_X[*Z]$ the
sheaf of meromorphic functions with poles on $Z$. Let us remark that
$\mathcal{O}_X[*Z]$ is smooth over $\mathcal{O}_X$. Given $\mc M\in\Mod(\D_X)$, we set
$\mc M[*Z]:= \mc M\overset{D}\otimes\mathcal{O}_X[*Z]$. One can prove that
$R\Gamma_{[X\setminus Z]}\mc M\simeq \mc M[*Z] $.

We denote by $X_\rea$ the real analytic manifold underlying $X$ and by
$\D_{X_\rea}$ the sheaf of linear differential operators with real
analytic coefficients. Furthermore, we denote by $\overline X$ the complex
conjugate manifold, in particular $\mathcal{O}_{\overline X}$ is the sheaf of
anti-holomorphic functions. Denote by $\Db_{X_\rea}$ the sheaf of
distributions on $X_\rea$ and, for a closed subset $Z$ of $X$, by
$\Gamma_Z(\Db_{X_\rea})$ the subsheaf of sections supported by
$Z$. One denotes by $\dbtxsa$ the presheaf of \emph{tempered
  distributions} on $X_{\rea}$ def\mbox{}ined by
$$\Op^c(\xsa)\ni U \longmapsto \dbtxsa(U):=\Gamma(X;\dbxr)\big/\Gamma_{X\setminus U}(X;\dbxr) \ .$$
In \cite{ks_indsheaves} it is proved that $\dbtxsa$ is a sheaf on $\xsa$. This
sheaf is well def\mbox{}ined in the category $\mod(\rho_!
\D_X)$. Moreover, for any $U\in\opxsac$, $\dbtxsa$ is
$\Gamma(U,\cdot)$-acyclic.

One def\mbox{}ines the complex of sheaves $\ot_\xsa \in D^b\bigr{\rho_!\D_X}$ 
of tempered holomorphic functions as
$$\ot_\xsa:=R\mathcal Hom_{\rho_! \mathcal{D}_{\overline X}}\bigr{\rho_!\mathcal{O}_{\overline X},\dbt_{X_{\rea}}}\ .$$

Let $\Omega^j_X$ be the sheaf of differential forms of degree $j$ and,
for sake of simplicity, let us write $\Omega_X$ instead of
$\Omega^{d_X}_X$.

Set
\begin{equation}
  \label{eq:drd}
  \omt X:=\rho_!\Omega_X\ou{\rho_!\mathcal{O}_X}\otimes \otxsa\ , \hspace{8mm}
  \Omega^{\dbt}_X:=\rho_!\Omega_X\ou{\rho_!\mathcal{O}_X}\otimes \dbt_\xsa\ ,
  \hspace{8mm}
  \Omega^{j,\dbt}_X:=\rho_!\Omega^j_X\ou{\rho_!\mathcal{O}_X}\otimes \dbt_\xsa \ .
\end{equation}

Let us now define the De Rham functors over $\D_X$, let $\mc M\in\bdc{\D_X}$,

\[
\mrm{DR}_{\D_X}\mc M:=\Omega\ou[L]{\D_X}\otimes\mc M\ ,\hspace{5mm}
\drt [X]{\mc M }:= \omt X\ou [L]{\rho_!\D_X}\otimes\rho_! \mc M\ ,\hspace{5mm}
\mrm{DR}^{\dbt}_{\D_X}\mc M:=\Omega^{\dbt}\ou[L]{\rho_!\D_X}\otimes\rho_!\mc M\ .
\]

Furthermore, let us define the De Rham complexes over $\mathcal{O}_X$, let $\mc
M\in\dmod X$,

\begin{equation}\label{eq:dro}
  \mrm{DR}_{\mathcal{O}_X}\mc M  :=  0\lra\mc
  M\overset{\nabla^{(0)}}\lra\Omega^1_X\ou{\mathcal{O}_X}\otimes \mc M\overset{\nabla^{(1)}}\lra\Omega^2_X\ou{\mathcal{O}_X}\otimes \mc M\lra
  \ldots \ ,
\end{equation}
\begin{equation}\label{eq:drodbt}
  \mrm{DR}^{\dbt}_{\mathcal{O}_X}\mc M  :=  0\lra\rho_!\mc
  M\overset{\nabla^{(0)}}\lra\Omega^{1,\dbt}_X\ou{\rho_!\mathcal{O}_X}\otimes\rho_!\mc M\overset{\nabla^{(1)}}\lra\Omega^{2,\dbt}_X\ou{\rho_!\mathcal{O}_X}\otimes\rho_!\mc M\lra 
  \ldots \ ,
\end{equation}
where $\nabla^{(j)}$ is defined by the action of vector fields on $\mc
M$.

The complex $\mrm{DR}_{\mathcal{O}_X}$ is an object in the category of
differential complexes. As we do not need the theory of differential
complexes, we do not recall it here and we refer to
\cite{maisonobe_sabbah} and to the references cited there. Moreover,
the objects and the results recalled here on the De Rham functors can
be stated in much more general settings, we refer to the bibliography.

Let $\bdc[\com-c]{\com_X}$ be the full subcategory of $\bdc{\com_X}$
whose objects have $\com$-constructible cohomology groups
(\cite{ks_som}).

\begin{prop}[\cite{bjork}, Proposition 2.2.10]\label{prop:drd_dro}
  Let $\mc M\in\bdc[h]{\D_X}$. The complexes $\mrm{DR}_{\mathcal{O}_X}\mc M$ and
  $\mrm{DR}_{\D_X}\mc M[-d_X]$ are isomorphic in the category $\bdc [\com-c]{\com_X}$.
\end{prop}

One can prove that the isomorphism of Proposition
\ref{prop:drd_dro} extends to
\begin{equation}
  \label{eq:drodbt_drddbt}
\mrm{DR}^{\dbt}_{\mathcal{O}_X}\mc M\simeq\mrm{DR}^{\dbt}_{\D_X}\mc M[-d_X]
\end{equation}
 as objects in $\bdc{\com_\xsa}$.




\begin{thm}[\cite{ks_indsheaves} Theorem 7.4.12, Theorem 7.4.1]\label{thm:indsheaves}
  \begin{enumerate}
  \item  Let $\mc L\in\dbdx[rh]$ and set $L:=\RH[\D_X]{\mc L}{\mathcal{O}_X}$. There
  exists a natural isomorphism in $\bdc{\com_\xsa}$
$$  \drt[X]{\mc L}\simeq\RH[\com_\xsa]L{\omt X}  \ . $$
\item Let $f:X\to Y$ be a holomorphic map and let $\N\in\dbd Y$. There is
  a natural isomorphism in $D^b(\com_\xsa)$
$$ \drt [X]{(\Dfinv f \mc N)}[d_X]\overset\sim\lra f^!(\drt [Y]{\mc N})[d_Y] \ .
$$
  \end{enumerate}
\end{thm}

We are now going to recall a conjecture of M. Kashiwara and
P. Schapira on constructibility of tempered holomorphic solutions of
holonomic $\D$-modules and the results we obtained on curves.

For sake of shortness, we set
$$ \solt(\M):=\RH[\rho_!\D_X]{\rho_!\mc M}{\ot}\in\bdc{\com_\xsa} \ .$$

\begin{conj}[\cite{ks_micro_indsheaves}]\label{conj:constr} Let
  $\mc M\in\bdc[h]{\D_X}$. Then $\solt(\M)\in\bdc{\com_\xsa}$ is $\rea$-constructible.
\end{conj}

In \cite{morando_existence_theorem}, we proved that Conjecture
\ref{conj:constr} is true on analytic curves.

\begin{thm}\label{thm:r-c_dim1}
Let $X$ be a complex curve and $\M\in\bdc[h]{\D_X}$. Then, 
$\solt(\mc M)$ is $\rea$-constructible.
\end{thm}

Let $\mc M\in\dbdx[coh]$. Set
$$ \mathbb D_X:=\mc Ext^{d_X}_{\D_X}(\mc M,\D_X\otimes_{\mathcal{O}_X}\Omega_X^{\otimes-1}) \ . $$

\begin{thm}
  \begin{enumerate}
  \item The functor $\mathbb D_X:\dmod[h]X^{op}\to\dmod[h]X$
    is an equivalence of categories.
  \item Let $\mc M\in\dbdx [h]$. Then,
$$  \solt(\mc M)\simeq\drt[X]{(\mathbb D_X\mc M)}  \ .$$
\end{enumerate}  
\end{thm}

\subsection{Elementary asymptotic decomposition of meromorphic connections}

In this subsection we are going to recall some fundamental results on
the asymptotic decomposition of meromorphic connections. The first
results on this subject were obtained by H. Majima (\cite{majima_lnm})
and C. Sabbah (\cite{sabbah_aif}). In particular, Sabbah proved that
any meromorphic connection admitting a \emph{good} formal decomposition
admits an asymptotic decompositions on small multisectors. Let us
recall that Sabbah conjectured that any meromorphic connection admits
a \emph{good} formal decomposition after a finite number of pointwise
blow-up and ramification maps (\cite{sabbah_ast}). Such a conjecture
was proved in the algebraic case by T. Mochizuki
(\cite{mochizuki1,mochizuki2}) and in the analytic case by K. Kedlaya
(\cite{kedlaya1,kedlaya2}). Let us also stress the fact that, in the
works of Sabbah, Mochizuki and Kedlaya, there is an essential
\emph{goodness} property of formal invariants which plays a central
role in the study of the formal and asymptotic decompositions.  As in
this paper we are not concerned with the formal decomposition of
meromorphic connections and since the goodness property is not needed
within the scope of our results, we are not going to give details on
them. In this way, we will avoid to go into technicalities unessential
in the rest of the paper. For this subsection we refer to
\cite{sabbah_ast,mochizuki2,hien_periods_manifolds}.

Let us start by recalling some results about integrable connections on an
analytic manifold $X$ of dimension $n$ with meromorphic poles on a
divisor $Z$. As in the rest of the paper we will just need the case
where $Z$ is a normal crossing hypersurface, from now on, we will suppose
such hypothesis.

Let $\Omega^j_X$ be the sheaf of $j$-forms on $X$. Let $\mc M$ be a
finitely generated $\mathcal{O}[*Z]$-module endowed with a $\com_X$-linear morphism
$\nabla:\mc M\to\Omega^1_X\otimes_{\mathcal{O}_X}\mc M$ satisfying the Leibniz
rule, that is to say, for any $h\in\mc O[*Z]$, $m\in\mc M$,
$\nabla(hm)=dh\otimes m+h\nabla m$. The morphism $\nabla$ induces
$\com_X$-linear morphisms $\nabla^{(j)}:\Omega^j_X\otimes_{\mathcal{O}_X}\mc
M\to \Omega^{j+1}_X\otimes_{\mathcal{O}_X}\mc M$.

\begin{df}
  A \emph{meromorphic flat connection on $X$ with poles along $Z$ }
is a locally free $\mathcal{O}[*Z]$-module
  of finite type $\mc M$ endowed with a $\com_X$-linear morphism
  $\nabla$ as above such that $\nabla^{(1)}\circ\nabla=0$.
\end{df}

For sake of shortness, in the rest of the paper, we will drop the
adjective ``\emph{flat}''. If there is no risk of confusion, given a
meromorphic connection $(\mc M,\nabla)$, we will simply denote it by
$\mc M$.

Let $\mc M_1, \mc M_2$ be two coherent $\mathcal{O}_X[*Z]$-modules, a morphism
$\phi:\M_1\to\M_2$ induces a morphism
$\phi':\Omega^1_X\otimes_{\mathcal{O}_X}\mc M_1\to\Omega^1_X\otimes_{\mathcal{O}_X}\mc
M_2$. A \emph{morphism of meromorphic connections} $(\mc
M_1,\nabla_1)\to(\mc M_2,\nabla_2)$ is given by a morphism
$\phi:\M_1\to\M_2$ of coherent $\mathcal{O}_X[*Z]$-modules such that
$\phi'\circ\nabla_1=\nabla_2\circ\phi$. We denote by $\mfr M(X,Z)$
the \emph{category of meromorphic connections with poles along} $Z$.

It is well known (see \cite{bjork}) that, if $Z$ is a normal crossing
hypersurface, the image through the functor
$\cdot\overset{D}\otimes\mathcal{O}[*Z]$ of the full subcategory of $\dmod[h]X$
of objects with singular support contained in $Z$, is equivalent to
the category of meromorphic connections with poles along $Z$. In
particular, if $\mc M\in\dmod[h]{X}$ satisfies $S(\mc M)$ is a normal
crossing hypersurface and $\mc M\simeq R\Gamma_{[X\setminus S(\mc
  M)]}\mc M\simeq\mc M\Dotimes\mathcal{O}[*S(\mc M)]$, then the morphism
$\nabla^{(0)}:\mc M\lra\Omega^1_X\ou{\mathcal{O}_X}\otimes\mc M$, defined in
\eqref{eq:dro}, gives rise to a meromorphic connection. A meromorphic
connection is said \emph{regular} if it is regular as a
$\D$-module. Furthermore the tensor product in $\mfr M(X,Z)$ is well
defined and it coincides with the tensor product of $\D$-modules. With
an abuse of language, given a holonomic $\D_X$-module $\mc M$ with
singular support contained in $Z$, we will call $\mc M$ a meromorphic
connection if $\mc M\simeq R\Gamma_{[X\setminus Z]}\mc M\simeq \mc
M\Dotimes\mathcal{O}_X[*Z]$.

Let us recall that, if $(\mc M,\nabla)\in\mfr M(X,Z)$ and $\mc M$ is
an $\mathcal{O}_X[*Z]$-module of rank $r$, then, in a given basis of local
sections of $\mc M$, we can write $\nabla$ as $d-A$ where $A$ is a
$r\times r$ matrix with entries in
$\Omega^1_X\otimes_{\mathcal{O}_X}\mathcal{O}_X[*Z]$. Now, let $X'$ be a complex
manifold and $f:X'\to X$ a holomorphic map. Let us suppose that
$Z':=f^{-1}(Z)$ has codimension $1$ everywhere in $X'$. Then, we can
define the inverse image $f^*\mc M$ of $\mc M$ on $X'$ with poles
along $Z'$. As an $\mathcal{O}_{X'}[*Z']$-module, it is $f^{-1}\mc M$ and the
matrix of the connection in a local base is $f^*A$. With a harmless
abuse of notation, we will write $\dfinv f\mc M$ for $f^*\mc M$.


Let us now introduce the elementary asymptotic decompositions.

Let us denote by $\widetilde X$ the real oriented blow-up of the
irreducible components of $Z$ and by $\pi:\widetilde X\to X$ the
composition of all these. Let us suppose that $Z$ is locally defined
by $x_1\cdot\ldots\cdot x_k=0$. Then, locally $\widetilde X\simeq
(S^1\times\rea_{\geq0})^k\times\com^{n-k}$. By a \emph{multisector} we
mean a set of the form $\ou[k]{j=1}\prod(I_j\times V_j)\times W$ where
$I_j\subset S^1$ is an open connected set, $V_j=[0,r_j[$ ($r_j>0$) and
$W\subset \com^{n-k}$ is an open polydisc. Let $\overline x_1, \ldots,
\overline x_n$ denote the antiholomorphic coordinates on $X$. Then
$\partial_{\overline x_j}$ acts on $\mc C^\infty_{\widetilde
  X}$. The sheaf on $\tilde X(Z)$ of holomorphic functions with
asymptotic development on $Z$, denoted $\mc A_{\widetilde X}$, is
defined as
$$ \mc A_{\widetilde X}:=\ou[k]{j=1}\bigcap\mrm{ker}\Big(\overline x_j\partial _{\overline x_j}:\mc C^\infty_{\widetilde
  X(D)}\to\mc C^\infty_{\widetilde
  X(D)}\Big)\cap\ou[n]{j=k+1}\bigcap \mrm{ker}\Big(\partial _{\overline x_j}:\mc C^\infty_{\widetilde
  X(D)}\to\mc C^\infty_{\widetilde
  X(D)}\Big)   \ .$$
The sections of $\mc A_{\widetilde X}$ are holomorphic functions on
$\widetilde X$ which admit an asymptotic development in the sense of
\cite{majima_lnm} (see also \cite{sabbah_ast}).

Given $\mc M\in\mfr M(X,Z)$, we set $\mc M_{\widetilde X}:=\mc
A_{\widetilde X}\ou{\pi^{-1}\mathcal{O}_X}\otimes\pi^{-1}\mc M$.

\begin{df}\label{df:elementary_model}
  \begin{enumerate}
\item Let $\phi$ be a local section of $\mathcal{O}_X[*Z]/\mathcal{O}_X$, we denote by
  $\mc L^\phi$ the meromorphic connection of rank $1$ whose matrix in
  a basis is $d\phi$.
  \item An \emph{elementary local model} $\mc M$ is a meromorphic
    connection isomorphic to a direct sum
$$\ou{\alpha\in A}\oplus\mc L^{\phi_\alpha}\otimes\mc R_\alpha  \ , $$
where $A$ is a finite set, $(\phi_\alpha)_{\alpha\in A}$ is a family
of local sections of $\mathcal{O}_X[*Z]/\mathcal{O}_X$ and $(\mc R_\alpha)_{\alpha\in
  A}$ is a family of regular meromorphic connections.
\item We say that $(\mc M,\nabla)\in\mfr M(X,Z)$ admits an elementary
  $\mc A$-decomposition if for any $\theta\in\pi^{-1}(Z)$ there exist
  an elementary local model $(\mc M^{el},\nabla^{el})$ and an
  isomorphism
$$  (\mc M_{\widetilde X,\theta},\nabla)\simeq(\mc M^{el}_{\widetilde X,\theta},\nabla^{el})  \ .$$
In particular, if $\mrm{rk}\,\mc M=r$, for any $\theta\in\pi^{-1}(Z)$, there exists $Y_\theta\in\mrm Gl(r,\mc
  A_\theta)$ such that the following diagram commutes
    $$\xymatrix{
      0\ar[r] &\mc M_{\widetilde {X},\theta}
      \ar[r]^{\nabla\phantom{abcde}}\ar[d]^{Y_\theta\cdot}& \mc
      M_{\widetilde {X},\theta}\ou{\pi^{-1}\mathcal{O}_X}\otimes
      \Omega^1_{\widetilde{X},\theta}\ar[r]^{\phantom{}\nabla}\ar[d]^{Y_\theta\cdot}&
      \mc M_{\widetilde {X},\theta}\ou{\pi^{-1}\mathcal{O}_X}\otimes
      \Omega^2_{\widetilde{X},\theta}\ar[r]^{\phantom{habcde}\nabla}\ar[d]^{Y_\theta\cdot}&
      \ldots\\ 
      0\ar[r] &\mc M^{el}_{\widetilde {X},\theta}
      \ar[r]^{\nabla^{el}\phantom{abcde}}& \mc M^{el}_{\widetilde
        {X},\theta}\ou{\pi^{-1}\mathcal{O}_X}\otimes
      \Omega^1_{\widetilde{X},\theta}\ar[r]^{\phantom{}\nabla^{el}}&\mc
      M^{el}_{\widetilde {X},\theta}\ou{\pi^{-1}\mathcal{O}_X}\otimes
      \Omega^2_{\widetilde{X},\theta}\ar[r]^{\phantom{aseabcde}\nabla^{el}}&
      \ldots \ .\\ 
    }$$
\end{enumerate}
\end{df}

Let us now choose local coordinates such that $Z$ is defined by the
equation $x_1\cdot\ldots\cdot x_k=0$. 
 By a \emph{ramification map fixing} $Z$ we mean a map
\begin{equation}\label{eq:ramification}
      \begin{array}{rccl}
        \rho_l: & \com^n & \lra & \com^n \\
        & (t_1\ldots,t_n) & \longmapsto & (t_1^l\ldots,t_k^l,t_{k+1}\ldots,t_n) \ ,
      \end{array}
    \end{equation}
    for some $l\in\integergz$.

    The proof of Theorem \ref{thm:asymptotic_lift} below is based on
    deep results of T. Mochizuki (\cite{mochizuki1, mochizuki2}),
    K. Kedlaya (\cite{kedlaya1,kedlaya2}) and C. Sabbah
    (\cite{sabbah_ast, sabbah_aif}), see also
    \cite{hien_periods,hien_periods_manifolds} for a concise
    exposition.

\begin{thm}\label{thm:asymptotic_lift}
  Let $(\mc M, \nabla)\in\mfr M(X,Z)$. For any $x_0\in Z$ there exist
  a neighbourhood $W$ of $x_0$ and a finite sequence of pointwise blow-ups
  above $x_0$, $\sigma:Y\to X$ such that $\sigma^{-1}(Z)$ is a normal
  crossing divisor and there exists a ramification map $\eta:X'\to Y$
  fixing $\sigma^{-1}(Z)$ such that $\dfinv{(\eta\circ\sigma)}\mc
  M|_W$ admits an elementary $\mc A$-decomposition.


\end{thm}

\section{$\rea$-preconstructibility of the tempered De Rham complex for
  holonomic  $\D$-modules} 
\sectionmark{Preconstructibility of tempered De Rham complex}

In this section we are going to prove the $\rea$-preconstructibility
of $\drt[X]{\mc M}:=\omt X\ou{\rho_!\D_X}{\otimes}\rho_!\mc M$ for $X$
a complex analytic manifold and $\mc M\in\bdc[h]{\D_X}$. Such a result
is a weaker version of Conjecture \ref{conj:constr} on the
$\rea$-constructibility of $\drt[X]{\mc M}$.

Recall that $F\in\bdc[]{\com_\xsa}$ is said \emph{$\rea$-preconstructible} if
for any $G\in\bdc[\rea-c]{\com_X}$ with compact support and any
$j\in\integer$,
$$ \dim_\com\mrm R^jHom_{\com_X}(G,F)<+\infty\ . $$

The general statement will be given and proved in Subsection
\ref{subsec:general_statement}. Roughly speaking, the proof is based
on an induction process on $\dim\,X$. The first step of the induction
follows from Theorem \ref{thm:r-c_dim1}. In Subsections
\ref{subsec:finiteness_support} and
\ref{subsec:finiteness_connections} we treat particular cases of the
inductive step.

\subsection{The case of modules supported on analytic sets}\label{subsec:finiteness_support}

In this subsection we prove the following

\begin{prop}\label{prop:finiteness_support}
  Let $X$ be a complex analytic manifold, $Z\neq X$ a closed analytic subset
  of $X$. Suppose that for any complex analytic manifold $Y$ with
  $\dim\,Y<\dim\,X$ and for any $\mc N\in\bdc[h]{\D_Y}$, $\drt[Y]{\mc
    N}$ is $\rea$-preconstructible. 
  Then, for any $\mc M\in\dbdx[h]$, $\drt[X]{(R\Gamma_{[Z]}\mc
    M)}$ is $\rea$-preconstructible.
\end{prop}

Before going into the proof of Proposition
\ref{prop:finiteness_support}, let us recall few facts on complex
analytic subsets of an analytic manifold $X$. Given an analytic set
$Z\subset X$, a point $z\in Z$ is said \emph{regular} if there exists
a neighbourhood $W$ of $z$ such that $Z\cap W$ is a closed submanifold
of $W$. It is well known that the set of regular points of a regular
analytic set $Z\subset X$, denoted $Z_{reg}$, is a dense open subset
of $Z$. Furthermore $Z\setminus Z_{reg}$ is a closed analytic set called
the \emph{singular part of $Z$}.

\begin{proof}
  We want to prove that for any $j>0$, $G\in\bdc[\rea-c]{\com_X}$,
  with compact support,

\begin{equation}  \label{eq:finiteness_support}
 \dim_\com R^j\ho[\com_X]{G}{\drt[X]{(R\Gamma_{[Z]}\mc M)}}<+\infty \ .
\end{equation}
As $\bdc[\rea-c]{\com_X}$ is generated by objects of the form
$\com_U$, for $U\in\Op^c(\xsa)$, it is sufficient to prove
\eqref{eq:finiteness_support} with $G\simeq\com_U$.

We have the following sequence of isomorphisms

\begin{eqnarray*}
  R\ho[\com_X]{\com_U}{\omt X\ou{\rho_!\D_X}\otimes\rho_!R\Gamma_{[Z]}\mc M}  &  \simeq  &  R\ho[\com_X]{\com_U}{\omt X\ou{\rho_!\D_X}\otimes(\rho_!\mc M\ou{\rho_!\mathcal{O}_X}\otimes\rho_! R\Gamma_{[Z]}\mathcal{O}_X)} \\
&  \simeq  &  R\ho[\com_X]{\com_U}{(\omt X\ou{\rho_!\mathcal{O}_X}\otimes\rho_! R\Gamma_{[Z]}\mathcal{O}_X)\ou{\rho_!\D_X}\otimes\rho_!\mc M}  \\
&  \simeq  &  R\ho[\com_X]{\com_U}{R\Ho[\com_\xsa]{\com_Z}{\omt X}\ou{\rho_!\D_X}\otimes\rho_!\mc M}  \\
&  \simeq  &  R\ho[\com_X]{\com_U\ou{\com_X}\otimes \com_Z}{\drt [X]{\mc M}}  \\
&  \simeq  &  R\ho[\com_X]{\com_{U\cap Z}}{\drt [X]{\mc M}} \ . \\
\end{eqnarray*}

We have used Proposition \ref{prop:tensor_commutation} in the
second isomorphism and Theorem \ref{thm:indsheaves}(i) in the third
isomorphism.

In particular, it is sufficient to prove \eqref{eq:finiteness_support}
with $G\simeq \com_{U\cap Z}$.

Let us now suppose that $Z\neq X$ is a closed submanifold of $X$. Set $d_Z:=\dim Z$ and $d_X:=\dim X$ Let us
denote by $i_Z:Z\to X$ the inclusion. By Theorem
\ref{thm:kashiwara_lemma}, there exists $\mc N\in\bdc[h]{\D_Z}$ such
that $\Dfinv {i_Z} R\Gamma_{[Z]}\mc M\simeq\mc N$. Furthermore, for
$V\in\Op^c(Z_{sa})$, one has $i_{Z!}\com_{Z,V}\simeq \com_{X,V}$. We can
now write the following sequence of isomorphisms

\begin{eqnarray*}
  R\ho[\com_X]{\com_{X,V}}{\omt X\otimes_{\D_X}\rho_!R\Gamma_{[Z]}\mc M}  &  \simeq  &  R\ho[\com_X]{i_{Z!}\com_{Z,V}}{\omt X\otimes_{\D_X}\rho_!R\Gamma_{[Z]}\mc M} \\
&  \simeq  &  R\ho[\com_Z]{\com_{Z,V}}{i_Z^{!}(\omt X\ou{\rho_!\D_X}\otimes\rho_!R\Gamma_{[Z]}\mc M)}  \\
&  \simeq  &  R\ho[\com_Z]{\com_{Z,V}}{\drt [Z]{(\Dfinv {i_Z}R\Gamma_{[Z]}\mc M)[d_X-d_Z]}} \\
&  \simeq  &  R\ho[\com_Z]{\com_{Z,V}}{\drt[Z]{\mc N}[d_X-d_Z]} \ ,
\end{eqnarray*}
where we have used Theorem \ref{thm:indsheaves} (ii) in the
third isomorphism.

Hence by hypothesis we have that, if $Z\neq X$ is a closed submanifold of
$X$, for any $j>0$, $G\in\bdc[\rea-c]{\com_X}$, with compact support,

\begin{equation} \label{eq:smooth_support}
 \dim_\com R^j\ho[\com_X]{G}{\drt[X] {(R\Gamma_{[Z]}\mc M)}}<+\infty \ .
\end{equation}

Let us treat now the case where $Z\neq X$ is a generic closed analytic
subset of $X$.

Consider the following distinguished triangle

$$  R\Gamma_{[Z\setminus Z_{reg}]}\mc M\lra R\Gamma_{[Z]}\mc M\lra R\Gamma_{[Z_{reg}]}\mc M\ou{+1}\lra  \ .$$

It follows that it is sufficient to prove
\eqref{eq:finiteness_support} for $Z=Z_{reg}$ and $Z=Z\setminus
Z_{reg}$.

As the statement is local, the case of $Z_{reg}$ is easily solved by
reducing to the case of a closed submanifold treated above by choosing
a convenient neighbourhood of $Z_{reg}$ where $Z_{reg}$ is closed.

The case of $Z\setminus Z_{reg}$ is treated by using induction on its
dimension.

\end{proof}

\subsection{The case of meromorphic connections}\label{subsec:finiteness_connections}

In this subsection we are going to prove the following 
\begin{prop}\label{prop:finiteness_connections}
  Let $X$ be a complex analytic manifold of dimension $n\geq2$. Suppose
  that for any complex analytic manifold $Y$ with $1\leq\dim\,Y<n$ and for
  any $\mc N\in\bdc[h]{\D_Y}$, $\drt[Y]{\mc N}$ is $\rea$-preconstructible. Then for
  any $\mc M\in\bdc[h]{\D_X}$ such that
  \begin{equation}
    \label{eq:localized}
    \mc M\simeq R\Gamma_{[X\setminus S(\mc M)]}\mc
    M \ ,
  \end{equation}
 $\drt[X]{\mc M}$ is $\rea$-preconstructible.
\end{prop}

Before going into the proof of Proposition
\ref{prop:finiteness_connections}, let us make some standard
reductions allowing us to add some hypothesis to the singular support
of $\mc M$.

\begin{lem}\label{lem:reduction}
  \begin{enumerate}
  \item Suppose that for any hypersurface $V\subset X$ and any $\mc
    M\in\bdc[h]{\D_X}$ such that $\mc M\simeq R\Gamma_{[X\setminus
      V]}\mc M $, $\drt[X]{\mc M}$ is $\rea$-preconstructible. Then for any $Z\neq X$
    closed analytic subset of $X$ and any $\mc M\in\bdc[h]{\D_X}$
    such that $\mc M\simeq R\Gamma_{[X\setminus Z]}\mc M $,
    $\drt[X]{\mc M}$ is $\rea$-preconstructible.
  \item Suppose that for any complex analytic manifold $Y$, any $\mc
    N\in\bdc[h]{\D_Y}$ satisfying $\mc N\simeq R\Gamma_{[Y\setminus
      S(\mc N)]}\mc N$ and such that $S(\mc N)$ is a normal crossing
    hypersurface, $\drt[Y]{\mc N}$ is $\rea$-preconstructible. Then, for any $\mc M\in\bdc[h]{\D_X}$
    satisfying \eqref{eq:localized} and such that $S(\mc M)$ is a
    hypersurface, $\drt[X]{\mc M}$ is $\rea$-preconstructible.
  \end{enumerate}
\end{lem}

\begin{proof}
  \emph{(i).} Locally $Z\neq X$ be a closed analytic subset of
  $X$. Then, locally, $Z=\{x\in X;\ f_1(x)=\ldots=f_l(x)=0\}$, for
  some $f_j\in\mathcal{O}_X$, $j=1,\ldots,l$. Set $Z_l:=\{x\in X;\ f_l(x)=0\}$
  and $\hat Z_l:=\{x\in X;\ f_1(x)=\ldots=f_{l-1}(x)=0\}$. We have the
  following distinguished triangle (see \cite[Lemma
  3.19]{kashiwara_dmod})
$$ R\Gamma_{[X\setminus Z]}\mc M\lra R\Gamma_{[X\setminus Z_l]}\mc
M\oplus R\Gamma_{[X\setminus\hat Z_l]}\mc M \lra R\Gamma_{[X\setminus
  \hat Z_l]}\mc M \overset{+1}\lra \ . $$
By induction, we reduce to the case of $Z$ a hypersurface which is
assumed in the hypothesis.

\emph{(ii).}  Suppose that $\mc M\in\dmod[h]X$. Let $Z:=S(\mc M)$ be a
hypersurface. Then
$$ R\Gamma_{[X\setminus Z]}\mc M\simeq \mathcal{O}_X[*Z]\Dotimes \mc M  \ . $$

In particular $\mc M$ is a meromorphic connection. It follows that
there exists a finite sequence of pointwise complex blow-up maps
$\pi:X'\to X$ such that $\pi|_{X'\setminus\pi^{-1}(Z)}$ is a
biholomorphism and $S(\Dfinv\pi\mc M)$ has normal crossings.

Now, given $G\in\bdc[\rea-c]{\com_X}$, we have the following sequence
of isomorphisms
  \begin{eqnarray*}
    R\ho[\com_X]{G}{\omt X\ou{\rho_!\D_X}\otimes R\Gamma_{[X\setminus Z]}\mc M}  &\simeq &   R\ho[\com_X]{G}{\omt X\ou{\rho_!\D_X}\otimes(\rho_!\mc M\ou{\rho_!\mathcal{O}_X}\otimes\rho_! R\Gamma_{[X\setminus Z]}\mathcal{O}_X)}  \\
    &  \simeq  &  R\ho[\com_X]{G}{(\omt X\ou{\rho_!\mathcal{O}_X}\otimes\rho_!R\Gamma_{[X\setminus Z]}\mathcal{O}_X)\ou{\rho_!\D_X}\otimes \mc M} \\
    &  \simeq  &  R\ho[\com_X]{G}{R\Ho[\com_\xsa]{\com_{X\setminus Z}}{\omt X}\ou{\rho_!\D_X}\otimes\rho_!\mc M}  \\
    &  \simeq  &  R\ho[\com_X]{G\ou{\com_X}\otimes \com_{X\setminus Z}}{\drt[X]{\mc M}}  \\
    &  \simeq  &  R\ho[\com_X]{R\pi_!\,\pi^!(G\otimes\com_{X\setminus Z})}{\drt[X]{\mc M}} \\
    &  \simeq  &  R\ho[\com_{X'}]{\pi^!(G\otimes\com_{X\setminus
        Z})}{\pi^!(\drt[X]{\mc M})} \\
    &  \simeq  &  R\ho[\com_{X'}]{\pi^!(G\otimes\com_{X\setminus Z})}{\drt[X']{(\Dfinv \pi\mc M)}} \ . \\
  \end{eqnarray*}

We have used Proposition \ref{prop:tensor_commutation} in the
second isomorphism and Theorem \ref{thm:indsheaves}(i) in the third
isomorphism.

The case of $\mc M\in\bdc[h]{\D_X}$ can be reduced to the case $\mc
M\in\dmod[h]X$ treated above by standard techniques.

\end{proof}

It follows that it is sufficient to prove Proposition
\ref{prop:finiteness_connections} adding the hypothesis that $Z:=S(\mc
M)$ is a normal crossing hypersurface. Then, Lemma \ref{lem:reduction}
will imply the general statement of Proposition
\ref{prop:finiteness_connections}. In particular, if $Z$ is an
hypersurface, \eqref{eq:localized} reads as $\mc M\simeq\mc
M\Dotimes\mathcal{O}[*Z]$ and we can treat $\mc M$ as a meromorphic connection.

Now, the rest of the proof of Proposition
\ref{prop:finiteness_connections} is divided in several lemmas which
can be roughly grouped in three steps. In the first part we prove the
statement for elementary models. Essentially, the method consists in
taking the inverse image of connections on curves and using Theorem
\ref{thm:r-c_dim1}. In the second part, we prove the statement for
meromorphic connections with an elementary $\mc
A$-decomposition 
using the first step and Theorem \ref{thm:asymptotic_lift}. In the
last part we study the properties of the tempered De Rham functor
under inverse image of complex blow-ups and ramification maps. In the
end of this Subsection we collect all the lemmas proved and we
conclude the proof of Proposition \ref{prop:finiteness_connections}.

We start by treating the case of elementary models. Recall Definition
\ref{df:elementary_model} (ii). First we need the following

\begin{lem}\label{lem:finiteness_inv_im}
  Let $X,Y$ be two complex manifolds, $f:X\to Y$ a holomorphic map, $\mc
  M\in\bdc[h]{\D_Y}$. If $\drt[Y]{\mc M}$ is $\rea$-preconstructible, then
  $\drt[X]{(\Dfinv f \mc M)}$ is.
\end{lem}

\begin{proof}
  Set $d_X:=\dim X$ and $d_Y:=\dim Y$. Given
  $G\in\bdc[\rea-c]{\com_X}$, we have the following isomorphisms
  \begin{eqnarray*}
    R\ho[\com_X]{G}{\omt X\ou{\rho_!\D_X}\otimes \rho_!\Dfinv f\mc M}  &\simeq &   R\ho[\com_X]{G}{f^!(\omt Y\ou{\rho_!\D_Y}\otimes\rho_!\mc M)[d_Y-d_X]}  \\
    &  \simeq  &  R\ho[\com_Y]{Rf_!G}{\omt X\ou{\rho_!\D_X}\otimes\rho_!\mc M[d_Y-d_X]} \ ,  \\
  \end{eqnarray*}
  where we have used Theorem \ref{thm:indsheaves}(ii) in the first
  isomorphism.

The conclusion follows.
\end{proof}

Now, we can suppose that $X\simeq \com^n$, $n\geq2$, $Z=\{(x_1,\ldots,
x_k)\in\com^n;\, x_1\cdot\ldots\cdot x_n=0\}$. Recall that, for
$\phi\in\mathcal{O}_X[*Z]/\mathcal{O}_X$, we denote by $\mc L^\phi$ the meromorphic
connection of rank $1$ whose matrix in a basis is $d\phi$.

\begin{lem}\label{lem:Lphi_case}
  Suppose that for any complex analytic manifold $Y$ with
  $1\leq\dim\,Y<n$ and for any $\mc N\in\bdc[h]{\D_Y}$, $\drt[Y]{\mc
    N}$ is $\rea$-preconstructible.  Let
  $\phi\in\frac{\mathcal{O}_X[*Z]}{\mathcal{O}_X}$, then $\drt [X] {\mc L^\phi}$ is
  $\rea$-preconstructible.
\end{lem}

\begin{proof}

If $n=1$, the conclusion follows by Theorem \ref{thm:r-c_dim1}.

First consider the map $f:\com^2\setminus\{x_2=0\}\to\com$,
$f(x_1,x_2)=\frac {x_1}{x_2}$. As $df\neq 0$, $\Dfinv f \mc
L^{1/x}\simeq\dfinv f \mc L^{1/x}\simeq \mc L^{x_2/x_1}$ and it can be
extended to $\com^2$. Using Theorem \ref{thm:r-c_dim1} and Lemma
\ref{lem:finiteness_inv_im}, we obtain that $\drt [\com^2]{\mc
  L^{x_2/x_1}}$ is $\rea$-preconstructible.

Given $\phi\in\frac{\mathcal{O}_{\com^n}[*Z]}{\mathcal{O}_{\com^n}}$, one checks easily
that there exist $p,q\in\com[x_1,\ldots,x_n]$ satisfying the
conditions
\begin{enumerate}
\item $\phi=\frac{p(x_1,\ldots,x_n)}{q(x_1,\ldots,x_n)}$,
\item if we define
$$  \begin{array}{rccl}
    g:&\com^n&\lra&\com^2\\
&(y_1,\ldots,y_n)&\longmapsto&(q(y_1,\ldots,y_n),p(y_1,\ldots,y_n)) \ ,
  \end{array}
$$
then  $S:=\{x\in\com^n;\ \mathrm{rk}\,dg(x)<2\}\neq X$.
\end{enumerate}


Set $\mc N:=\Dfinv
  g\mc L^{x_2/x_1}$. Clearly $H^0\mc N\simeq\dfinv
  g\mc L^{x_2/x_1}\simeq \mc L^\phi$. 

  Now, by Lemma \ref{lem:finiteness_inv_im} and the
  $\rea$-preconstructibility of $\drt [\com^2]{\mc L^{x_2/x_1}}$
  (resp. by Proposition \ref{prop:finiteness_support}) we have that
  $\drt[\com^n]{\mc N}$ (resp. $\drt[\com^n]{R\Gamma_{[S]}\mc N}$) is
  $\rea$-preconstructible. Hence, by the following distinguished
  triangle
$$ R\Gamma_{[S]}\mc N\lra\mc N\lra R\Gamma_{[X\setminus S]}\mc N\overset{+1}\lra \ , $$
we have that $\drt[\com^n]{R\Gamma_{[X\setminus S]}\mc N}$ is
$\rea$-preconstructible. Remark that, as $f$ is smooth outside $S$,
$H^j\mc N$ has support in $S$, for $j\geq1$. In particular,
$R\Gamma_{[X\setminus S]}H^0\mc N\simeq R\Gamma_{[X\setminus S]}\mc
N$. Hence $\drt[\com^n]{R\Gamma_{[X\setminus S]}H^0\mc N}$ is
$\rea$-preconstructible too. Furthermore, by Proposition
\ref{prop:finiteness_support}, $\drt[\com^n]{R\Gamma_{[S]}H^0\mc N}$
is $\rea$-preconstructible. Using the distinguished triangle
$$ R\Gamma_{[S]}H^0\mc N\lra H^0\mc N\lra R\Gamma_{[X\setminus S]}H^0\mc N\overset{+1}\lra \ , $$
we have that $\drt[\com^n]{H^0\mc N}$ is $\rea$-preconstructible and
the statement is proved.


\end{proof}

We conclude the first step of the proof Proposition
\ref{prop:finiteness_connections} with the following

\begin{lem}\label{lem:finiteness_good_models}
  Let $\mc M\in\dbdx[h]$ and $\mc R\in\dbdx[rh]$. If $\drt [X]{\mc M}$
  is $\rea$-preconstructible, then $\drt[X]{(\mc M\Dotimes\mc R)}$
  is. In particular, if for any complex analytic manifold $Y$ with
  $1\leq\dim\,Y<n$ and for any $\mc N\in\bdc[h]{\D_Y}$, $\drt[Y]{\mc
    N}$ is $\rea$-preconstructible, then for any elementary model
  $\mc M$, $\drt[X]{\mc M}$ is $\rea$-preconstructible.
\end{lem}

\begin{proof}
  Let $L:=R\Ho[\D_X]{\mc R}{\mathcal{O}_X}\in\bdc[\rea-c]{\com_X}$. For
  $G\in\bdc[\rea-c]{\com_X}$, consider the following sequence of
  isomorphisms
 \begin{eqnarray*}
   R\ho[\com_X]{G}{\omt X\ou{\rho_!\D_X}\otimes \rho_!(\mc M\Dotimes\mc R)}  &\simeq &   R\ho[\com_X]{G}{(\omt X\ou{\rho_!\mathcal{O}_X}\otimes\rho_!\mc R)\ou{\rho_!\D_X}\otimes\mc M} \\
&\simeq &   R\ho[\com_X]{G}{R\Ho[\com_\xsa]L{\omt X}\ou{\rho_!\D_X}\otimes\rho_!\mc M} \\
&\simeq &   R\ho[\com_X]{G}{R\Ho[\com_\xsa]L{\drt[X]{\mc M}}} \\
&\simeq &   R\ho[\com_X]{G\ou{\com_X}\otimes L}{\drt[X]{\mc M}}  \ .\\
  \end{eqnarray*}

  In the previous series of isomorphisms we have used Proposition
  \ref{prop:tensor_commutation} in the first isomorphism and Theorem
  \ref{thm:indsheaves}(i) in the second isomorphism. Hence, the first
  part of the statement is proved.

  To conclude the proof, it is sufficient to combine the first part of
  the statement with Lemma \ref{lem:Lphi_case}.
\end{proof}

Now, we go into the second step of the proof of Proposition
\ref{prop:finiteness_connections}. We consider meromorphic connections
with an elementary $\mc A$-decomposition.

\begin{lem}\label{lem:finiteness_good_decomposition}
  Suppose that for any complex analytic manifold $Y$ with
  $1\leq\dim\,Y<n$ and for any $\mc N\in\bdc[h]{\D_Y}$, $\drt[Y]{\mc
    N}$ is $\rea$-preconstructible. Let $\mc M\in\dmod[h]X$ be such
  that
  \begin{enumerate}
  \item $S(\mc M)$ is a normal crossing hypersurface,
  \item $\mc M\simeq R\Gamma_{[X\setminus S(\mc M)]}\mc M$,
  \item $\mc M$ admits an elementary $\mc A$-decomposition as a
    meromorphic connection.  
  \end{enumerate}
Then $\drt[X]{\mc M}$ is $\rea$-preconstructible.
\end{lem}

\begin{proof}

We want to prove that for any $G\in\bdc[\rea-c]{\com_X}$ with compact support, $j\in\integer$
\begin{equation}
  \label{eq:finiteness}
  \dim R^j\ho[\com_\xsa]G{\drt[X]{\mc M}}<+\infty\ .
\end{equation}

First, let us remark that it is sufficient to prove
\eqref{eq:finiteness} for $G=\com_U$ for $U\in\Op^c(\xsa)$. Then, we
have the following sequence of isomorphisms

  \begin{eqnarray*}
    R\ho[\com_X]{\com_U}{\omt X\ou{\rho_!\D_X}\otimes R\Gamma_{[X\setminus Z]}\mc M}  &\simeq &   R\ho[\com_X]{\com_U}{\omt X\ou{\rho_!\D_X}\otimes\rho_!(\mc M\ou{\mathcal{O}_X}\otimes \mathcal{O}_X[*Z])}  \\
    &  \simeq  &  R\ho[\com_X]{\com_U}{(\omt X\ou{\rho_!\mathcal{O}_X}\otimes\rho_!\mathcal{O}_X[*Z])\ou{\rho_!\D_X}\otimes \mc M} \\
    &  \simeq  &  R\ho[\com_X]{\com_U}{R\Ho[\com_\xsa]{\com_{X\setminus Z}}{\omt X}\ou{\rho_!\D_X}\otimes\rho_!\mc M}  \\
    &  \simeq  &  R\ho[\com_X]{\com_U\ou{\com_X}\otimes \com_{X\setminus Z}}{\drt[X]{\mc M}}   \ .\\
  \end{eqnarray*}

  Hence, it is sufficient to prove \eqref{eq:finiteness} for
  $G=\com_V$, $V\in\Op^c(\xsa)$, $V\subset X\setminus Z$.

  In particular, for any $V\subset X\setminus Z$, there exists a
  finite family of open multisectors $\{S_j\}_{j\in J}$ such that
  $V\subset\ou{j\in J}\cup S_j$.

  Now, let $\Omega_X^\bullet$ be the complex of differential forms on
  $X$. Recall definition \eqref{eq:drodbt} and the isomorphism
  \eqref{eq:drodbt_drddbt}. Since, for $U\in\Op^c(\xsa)$, $\dbt$ is
  $\Gamma(U,\cdot)$-acyclic, we have that
  $\Gamma(U,D\!R^{\dbt}_{\D_X}\mc M)$ is quasi-isomorphic to the
  complex

$$ 0\lra\rho_!\mc M(U)
\overset{\nabla^{(0)}}\lra
\rho_!\mc M\ou[]{\rho_!\mathcal{O}_X}{\otimes}\rho_!\Omega^1_X\ou[]{\rho_!\mathcal{O}_X}{\otimes}\dbt (U) 
\overset{\nabla^{(1)}}\lra
\rho_!\mc M\ou[]{\rho_!\mathcal{O}_X}{\otimes}\rho_!\Omega^2_X\ou[]{\rho_!\mathcal{O}_X}{\otimes}\dbt (U) 
\overset{\nabla^{(2)}}\lra
\ldots \ .
$$

Suppose that $\mathrm{rk}\mc M=r$. As $\mc M$ has an elementary $\mc
A$-decomposition, for any small enough multisector $S$ there exists
$Y_S\in\mathrm Gl(r,\mc A(S))$ and an elementary model $(\mc
M^{el},\nabla^{el})$ giving a quasi-isomorphism of complexes

$$
\xymatrix{
0\ar[r] & \mc M(S)\ar[r]^{\nabla^{(0)}\phantom{aaaaa}}\ar[d]^{Y_S} &
\mc M\ou[]{\mathcal{O}_X}{\otimes}\Omega^1_X\ou[]{\mathcal{O}_X}{\otimes}\mc A
(S)
\ar[r]^{\nabla^{(1)}}\ar[d]^{Y_S}
&
\mc M\ou[]{\mathcal{O}_X}{\otimes}\Omega^2_X\ou[]{\mathcal{O}_X}{\otimes}\mc A
(S)
\ar[r]^{\phantom{aaaaaaaa}\nabla^{(2)}}\ar[d]^{Y_S}& 
\ldots
\\
0\ar[r] & \mc M^{el}(S)\ar[r]^{\nabla^{el,(0)}\phantom{aaaaaa}} &
\mc M^{el}\ou[]{\mathcal{O}_X}{\otimes}\Omega^1_X\ou[]{\mathcal{O}_X}{\otimes}\mc A
(S)
\ar[r]^{\nabla^{el,(1)}}&
\mc M^{el}\ou[]{\mathcal{O}_X}{\otimes}\Omega^2_X\ou[]{\mathcal{O}_X}{\otimes}\mc A
(S)
\ar[r]^{\phantom{aaaaaaaaaa}\nabla^{el,(2)}}&
\ldots\ .
}
$$

For sake of shortness, in the diagram above, we have omitted the
notation of Definition \ref{df:elementary_model} relative to the real
blow-up $\widetilde X$.

Now, let $U\in\Op^c(\xsa)$ be contained in a sufficiently small
multisector. As the restriction of $Y_S$ to $U$ respects $\dbt(U)$, we
have the following quasi-isomorphism of complexes

$$
\xymatrix{
0\ar[r] & \rho_!\mc M(U)\ar[r]^{\nabla^{(0)}\phantom{aaaaaaaa}}\ar[d]^{Y_S|_U} &
\rho_!\mc M\ou[]{\rho_!\mathcal{O}_X}{\otimes}\rho_!\Omega^1_X\ou[]{\rho_!\mathcal{O}_X}{\otimes}\dbt
(U)
\ar[r]^{\nabla^{(1)}}\ar[d]^{Y_S|_U}
&
\rho_!\mc M\ou[]{\rho_!\mathcal{O}_X}{\otimes}\rho_!\Omega^2_X\ou[]{\rho_!\mathcal{O}_X}{\otimes}\dbt
(U)
\ar[r]^{\phantom{aaaaaaaaaaa}\nabla^{(2)}}\ar[d]^{Y_S|_U}& 
\ldots
\\
0\ar[r] & \rho_!\mc M^{el}(U)\ar[r]^{\nabla^{el,(0)}\phantom{aaaaaaaa}} &
\rho_!\mc M^{el}\ou[]{\rho_!\mathcal{O}_X}{\otimes}\rho_!\Omega^1_X\ou[]{\rho_!\mathcal{O}_X}{\otimes}\dbt
(U)
\ar[r]^{\nabla^{el,(1)}}&
\rho_!\mc
M^{el}\ou[]{\rho_!\mathcal{O}_X}{\otimes}\rho_!\Omega^2_X\ou[]{\rho_!\mathcal{O}_X}{\otimes}\dbt
(U)
\ar[r]^{\phantom{aaaaaaaaaaaa}\nabla^{el,(2)}}&
\ldots \ .
}
$$

The second row of the above diagram is isomorphic to
$\D\!R^{\dbt}_{\mathcal{O}_X}(\mc M^{el})$.

By the isomorphism \eqref{eq:drodbt_drddbt}, we have 
\begin{equation}
  \label{eq:iso_drdbt_U}
  D\!R^{\dbt}_{\D_X}\mc M(U)\simeq D\!R^{\dbt}_{\D_X}\mc M^{el}(U) \ .
\end{equation}

Sheafifying \eqref{eq:iso_drdbt_U} we have that 
\begin{equation}
  \label{eq:iso_drdbt}
  \big(D\!R^{\dbt}_{\D_X}\mc M\big)_S\simeq\big( D\!R^{\dbt}_{\D_X}\mc M^{el}\big)_S \ .
\end{equation}

Remark that the isomorphisms \eqref{eq:iso_drdbt} depend on $S$ and
they can't be glued in a global isomorphism.

Now, taking the solutions of the Cauchy-Riemann system (i.e. applying
the functor $\cdot\ou[]{\rho_!\D_{\overline X}}{\otimes}\rho_!\mathcal{O}_{\overline X}$ to
\eqref{eq:iso_drdbt}) we have
\begin{equation}
  \label{eq:iso_drt}
  \big(\drt[X]{\mc M}\big)_S\simeq\big(\drt[X]{\mc M^{el}}\big)_S \ .
\end{equation}

The $\rea$-preconstructibility of the right hand side of
\eqref{eq:iso_drt} following from Lemma
\ref{lem:finiteness_good_models}, the proof is complete.


\end{proof}

Let us now study the behaviour of $\rea$-preconstructibility of the
tempered De Rham complex under inverse image in the case of a
composition of complex blow-ups and ramification maps.

\begin{lem}\label{lem:finiteness_blow-up}
  Let $Z\neq X$ be a closed analytic subset of $X$ and $\pi:X'\to X$
  be a composition of pointwise complex blow-ups above $Z$ and ramification maps
  fixing $Z$. Let $\mc M\in\dbdx [h]$ be such that $\mc M\simeq
  R\Gamma_{[X\setminus Z]}\mc M$. If $\drt [X']{(\Dfinv \pi\mc M)}$ is
  $\rea$-preconstructible, then $\drt[X]{\mc M}$ is.
\end{lem}

\begin{proof}
  As concern a complex blow-up, the proof goes exactly as that of
  Lemma \ref{lem:reduction} \emph{(ii)}. As concern a ramification map
  $\rho_l$ as given in \eqref{eq:ramification}, the proof goes as in
  Lemma \ref{lem:reduction} \emph{(ii)} taking care of decomposing the
  sheaf $G$ on the elements of a finite covering on which $\rho_l$ is
  an isomorphism.

\end{proof}

\emph{Proof of Proposition \ref{prop:finiteness_connections}}

Let $\mc M\in\bdc[h]{\D_X}$ be such that $\mc M\simeq
R\Gamma_{[X\setminus S(\mc M)]}\mc M$. By Lemma \ref{lem:reduction} we
can suppose that $S(\mc M)$ is a normal crossing
hypersurface. Furthermore, by standard techniques, we can assume that $\mc
M\in\dmod[h]X$. In particular $\mc M$ can be considered as a
meromorphic connection.

By Theorem \ref{thm:asymptotic_lift}, there exists a finite sequence
of complex blow-ups and ramification maps $\pi$ such that, $\dfinv
\pi\mc M$ admits an elementary $\mc A$-decomposition. It follows, by
Lemma \ref{lem:finiteness_blow-up}, that we can suppose that $\mc M$
admits an elementary $\mc A$-decomposition.

Then, we can conclude by Lemma
\ref{lem:finiteness_good_decomposition}.



 

\qed

\subsection{The general statement}\label{subsec:general_statement}

Let us recall that, given $\mc M\in\bdc[h]{\D_X}$,

\[
\drt [X]{\mc M }:= \omt X\ou [L]{\rho_!\D_X}\otimes\rho_! \mc M\ .
\]

Moreover, $F\in\bdc[]{\com_\xsa}$ is said $\rea$-preconstructible if
for any $G\in\bdc[\rea-c]{\com_X}$ with compact support and any
$j\in\integer$,
$$ \dim_\com\mrm R^jHom_{\com_X}(G,F)<+\infty\ . $$

We can now state and prove

\begin{thm}
  Let $X$ be a complex analytic manifold, $\mc M\in\dbdx[h]$. Then
  $\drt [X]{\mc M}\in\bdc[X]{\com_\xsa}$ is $\rea$-preconstructible.
\end{thm}

\begin{proof}
Let us use induction on the dimension of $X$.

If $X$ is a curve, then the result follows at once from Theorem
\ref{thm:r-c_dim1}.

Now suppose that $\dim\,X>1$ and that the inductive hypothesis
holds. That is, for any complex analytic manifold $Y$ with
$1\leq\dim\,Y<\dim\,X$ and for any $\mc N\in\bdc[h]{\D_Y}$,
$\drt[Y]{\mc N}$ is $\rea$-preconstructible. Let us prove that for any $\mc
M\in\bdc[h]{\D_X}$, $\drt[X]{\mc M}$ is $\rea$-preconstructible.

By using the distinguished triangle 

$$R\Gamma_{[S(\mc M)]}\mc M\lra\mc M \lra R\Gamma_{[X\setminus S(\mc
  M)]}\mc M\overset{+1}\lra  \ ,$$

it is enough to prove the statement for $R\Gamma_{[S(\mc M)]}\mc M$
and for $R\Gamma_{[X\setminus S(\mc M)]}\mc M$. The conclusion 
follows by Propositions \ref{prop:finiteness_support} and
\ref{prop:finiteness_connections}


\end{proof}


\addcontentsline{toc}{section}{\textbf{References}}
\bibliography{biblio}{}
\bibliographystyle{abbrv}

\vspace{10mm}

{\small{\sc Giovanni Morando

Dipartimento di Matematica Pura ed Applicata,

Universit{\`a} degli Studi di Padova,

Via Trieste 63, 35121 Padova, Italy.}

E-mail address: $\texttt{gmorando@math.unipd.it}$}

\end{document}